\documentclass[11pt, letter]{amsart}
\usepackage{hyperref}
\hypersetup{nesting=true,debug=true,naturalnames=true}
\usepackage{amssymb,enumerate,csquotes}
\usepackage[capitalise, nameinlink]{cleveref}
\usepackage[all]{xy}
\makeatletter
\@namedef{subjclassname@2010}{%
  \textup{2010} Mathematics Subject Classification}
\makeatother

\let\<\langle
\let\>\rangle
\newcommand{\doi}[1]{doi:\,\href{http://dx.doi.org/#1}{#1}}

\let\uml\"

\newtheorem{theorem}{Theorem}[section]
\newtheorem{prop}[theorem]{Proposition}
\newtheorem{corollary}[theorem]{Corollary}
\newtheorem{lemma}[theorem]{Lemma}
\newtheorem{mainthm}{Theorem}

\theoremstyle{definition}
\newtheorem{definition}[theorem]{Definition}
\newtheorem{rem}[theorem]{Remark}
\newtheorem{example}[theorem]{Example}

\numberwithin{equation}{section}
\newcommand\N{{\mathbb{N}}}
\newcommand\C{{\mathbb{C}}}
\newcommand\Q{{\mathbb{Q}}}
\newcommand\Z{{\mathbb{Z}}}

\newcommand{\U}{\widehat{\mathcal{U}}}
\newcommand{\G}{\widehat{\mathcal{GL}}}
\hypersetup{
	pdfauthor={Apurva Seth, Prahlad Vaidyanathan},
	colorlinks,
	linkcolor=blue,          % color of internal links
	filecolor=magenta,      % color of file links
	urlcolor=cyan           % color of external links
}

\begin{document}
\title[AF-algebras and Rational Homotopy Theory]{AF-algebras and Rational Homotopy Theory}
\author[Apurva Seth, Prahlad Vaidyanathan]{Apurva Seth, Prahlad Vaidyanathan}
\address{Department of Mathematics\\ Indian Institute of Science Education and Research Bhopal\\ Bhopal ByPass Road, Bhauri, Bhopal 462066\\ Madhya Pradesh. India.}
%\address{Department of Mathematics\\ Indian Institute of Science Education and Research Bhopal\\ Bhopal ByPass Road, Bhauri, Bhopal 462066\\ Madhya Pradesh. India.}
\email{apurva17@iiserb.ac.in, prahlad@iiserb.ac.in}
\date{}
\begin{abstract}
We give a procedure to compute the rational homotopy groups of the group of quasi-unitaries of an AF-algebra. As an application, we show that an AF-algebra is K-stable if and only if it is rationally K-stable.
\end{abstract}
\subjclass[2010]{Primary 46L85; Secondary 46L80}
\keywords{Nonstable K-theory, C*-algebras}
\maketitle
\parindent 0pt

Given a unital C*-algebra $A$, the topological group $\mathcal{U}_n(A)$ of unitary $n\times n$ matrices over $A$ carries a great deal of information about $A$. In particular, the homotopy groups of $\mathcal{U}_n(A)$ converge (in a suitable sense) to the $K$-theory groups of $A$. Therefore, the study of these homotopy groups is important to understand finer $K$-theoretic information about $A$, and is termed nonstable $K$-theory. \\

Rieffel introduced the study of these groups in \cite{rieffel2}, and Thomsen \cite{thomsen} extended this to non-unital C*-algebras by developing the concept of quasi-unitaries. This allowed him to prove some functorial properties of these groups and also compute them in certain cases. Most of his calculations (and those of Rieffel) relied on an interesting property that certain infinite dimensional C*-algebras possess, namely that these homotopy groups are naturally isomorphic to the corresponding $K$-theory groups of the C*-algebra - a property called \emph{$K$-stability} (See \cref{defn:k_stable}). \\

If a C*-algebra $A$ is $K$-stable, then calculating the homotopy groups $\pi_j(\mathcal{U}_n(A))$ amounts to calculating the $K$-theory groups $K_{j+1}(A)$. However, in the absence of $K$-stability, we do not, as yet, have any good tools to calculate these homotopy groups. Even in the simplest of cases, when $A$ is $\C$, the algebra of complex numbers, the groups $\pi_j(\mathcal{U}_n(\C))$ are naturally related to those of spheres, and are not known for many values of $j$ and $n$. It is primarily to remedy this difficulty, that topologists introduced rational homotopy theory, which is where we turn to in this paper. \\

Upto rationalization, the homotopy groups of $\mathcal{U}_n(\C)$ are well-known (\cref{examplefinite}). Furthermore, maps between finite dimensional C*-algebras are, upto unitary equivalence, well understood. Our first main result shows that one can leverage these facts to compute the rational homotopy groups of the group of quasi-unitaries of an AF-algebra.

\begin{mainthm}\label{mainthm:one}
Let $A$ be an AF-algebra, and $\U(A)$ denote the group of quasi-unitaries of $A$. Given a Bratteli diagram describing $A$ and a positive integer $m$, there is a diagram $\mathcal{D}_m(A)$ of $\Q$-vector spaces whose inductive limit is the group $\pi_m(\U(A))\otimes \Q$.
\end{mainthm}

Our next result is a characterization of K-stability for an AF-algebra. We say that a C*-algebra $A$ is \emph{rationally} $K$-stable if the rational homotopy groups of the group of quasi-unitaries of $A$ are naturally isomorphic to those of $M_n(A)$ for all $n\geq 1$. As a consequence of \cref{mainthm:one}, we show that

\begin{mainthm}\label{mainthm:two}
For an AF-algebra $A$, the following are equivalent:
\begin{enumerate}
\item $A$ is $K$-stable.
\item $A$ is rationally $K$-stable.
\item $A$ has no non-zero finite dimensional representations.
\end{enumerate}
\end{mainthm}

Furthermore, we show that this property can also be read off from a Bratteli diagram describing the algebra.

\section{Preliminaries}
 
We begin by reviewing the work of Thomsen in constructing the nonstable K-groups associated to a C*-algebra. For the proofs of all the facts mentioned below, the reader may refer to \cite{thomsen}. \\

Let $A$ be a $C^*$-algebra (not necessarily unital). Define an associative composition $\cdot$ on $A$ by
\[
a\cdot b=a+b-ab
\]
An element $a\in A$ is said to be quasi-invertible if there exists $b\in A$ such that $a\cdot b = b\cdot a = 0$, and we write $\G(A)$ for the set of all quasi-invertible elements in $A$. An element $u\in A$ is said to be a quasi-unitary if $u\cdot u^{\ast} = u^{\ast}\cdot u = 0$, and we write $\U(A)$ for the set of all quasi-unitary elements in $A$. \\

If $B$ is a unital $C^*$-algebra, we write $\mathcal{GL}(B)$ for the group of invertibles in $B$ and $\mathcal{U}(B)$ for the group of unitaries in $B$. 

\begin{lemma}\label{lem:quasi_unitary_ideal}
Let $B$ be a unital C*-algebra and $A\subset B$ be a closed two-sided ideal of $B$. Then $a\in A$ is quasi-invertible if and only if $1-a \in \mathcal{GL}(B)$; and similarly, $u\in A$ is a quasi-unitary if and only if $(1-u) \in \mathcal{U}(B)$
\end{lemma}
Since $A$ can be thought of as a closed ideal in its unitization $A^+$, it follows that $\G(A)$ is open in $A$, $\U(A)$ is closed in $A$, and they both form topological groups. Furthermore, the map $r: \G(A)\to \U(A)$ given by 
\[
r(a) := 1-(1-a)((1-a^*)(1-a))^{-1/2}
\]
is a strong deformation retract, and hence a homotopy equivalence. \\

For elements $u,v\in \U(A)$, we write $u\sim v$ if there is a continuous function $f:[0,1]\to \U(A)$ such that $f(0) = u$ and $f(1) = v$. We write $\U_0(A)$ for the set of $u\in \U(A)$ such that $u\sim 0$. Note that $\U(A)$ and $\U_0(A)$ are both topological groups with a common base point $0$, and that $\U_0(A)$ is connected. We now define the two functors we are interested in.
  
\begin{definition}
Let $A$ be a $C^*$-algebra (not necessarily unital) and $k\geq 0$ and $m\geq 1$ be integers. Define
\begin{equation*}
\begin{split}
G_k(A) &:=\pi_k(\U(A)), \text{ and } \\
F_m(A) &:= \pi_m(\U_0(A))\otimes \Q \cong G_m(A)\otimes \Q
\end{split}
\end{equation*}
\end{definition}

\begin{rem}\label{rem:quasi_unitary_unital}
Note that if $A$ is unital, the map $\U(A) \to \mathcal{U}(A)$ given by $u \mapsto (1-u)$ induces natural isomorphisms
\[
G_k(A) \cong \pi_k(\mathcal{U}(A)) \text{ and } F_m(A) \cong \pi_m(\mathcal{U}_0(A))\otimes \Q
\]
\end{rem}

Recall \cite[Definition 21.1.1]{blackadar} that a homology theory on the category of $C^{\ast}$-algebras is a sequence $\{h_n\}$ of covariant, homotopy invariant functors from the category of $C^*$-algebras to the category of abelian groups such that, if $0 \to J\xrightarrow{\iota} B\xrightarrow{p} A\to 0$ is a short exact sequence of $C^*$-algebras, then for each $n \in \N$, there exists a connecting map $\partial : h_n(A)\to h_{n-1}(J)$, making the following sequence exact
\[
\ldots \xrightarrow{\partial} h_n(J)\xrightarrow{h_n(\iota)} h_n(B) \xrightarrow{h_n(p)} h_n(A)\xrightarrow{\partial} h_{n-1}(J)\to \ldots
\]
and furthermore, $\partial$ is natural with respect to morphisms of short exact sequences. Furthermore, we say that a homology theory $\{h_n\}$ is continuous if, whenever $A = \lim A_i$ is an inductive limit in the category of C*-algebras, then $h_n(A) = \lim h_n(A_i)$ in the category of abelian groups. The proof of the next theorem is contained in \cite[Proposition 2.1]{thomsen} and \cite[Theorem 4.4]{handelman}.

\begin{prop}
For integers $k\geq 0$ and $m\geq 1$, $G_k$ and $F_m$ are continuous homology theories.
\end{prop}

Note that every functor in a homology theory is additive. We conclude this discussion with a calculation of $F_m(A)$ when $A$ is a finite dimensional C*-algebra. \\

Recall that a space $X$ is said to be of finite rational type if $H^i(X;\Q)$ is finite dimensional for each $i$. Sullivan \cite{sullivan} and Bousfield-Gugenheim \cite{bousfield} showed that the rational homotopy category of nilpotent spaces of finite rational type is equivalent to the homotopy category of commutative, associative, differential graded rational algebras with minimal models of finite type. Furthermore, the rational homotopy groups of the space may be computed from the corresponding algebra. For $H$-spaces, this theorem simplifies to the following theorem.

\begin{theorem}[Sullivan]\label{sullivan}
Let $X$ be connected $H$-space of finite type. Then its minimal model is of the form $(\wedge V,0)$. In particular,
\[
H^{\ast}(X;\mathbb{Q})=\wedge V\text{ and } \pi_{\ast}(X)\otimes\mathbb{Q}\cong V^{\ast}
\]
Furthermore, the construction of $V$ is functorial and the above isomorphism is natural. 
\end{theorem}

Note that, if $V$ is a non-negatively graded vector space, then $\wedge V$ denotes the free, commutative, graded algebra generated by $V$. Furthermore, if $\{v_j,j\in J\}$ is a basis of $V$, we shall write $\wedge (v_j)$ for $\wedge V$. 

\begin{example}\label{examplefinite}
For $n \in \N$,
\[
H^{\ast}(\mathcal{U}_n(\C);\Q) \cong \wedge(x_1,x_3,\hdots x_{2n-1})
\]
where $x_i$ has degree $i$, by \cite[III, Corollary 3.11]{mimura}. It follows by \cref{sullivan} that
\[
\pi_m(\mathcal{U}_n(\C))\otimes \Q = \begin{cases}
\Q &: 1\leq m\leq 2n-1, m \text{ odd} \\
0 &: \text{ otherwise}
\end{cases}
\]
Let $A=M_n(\C)$, then by \cref{rem:quasi_unitary_unital},
\[
F_m(A) \cong \pi_m(\mathcal{U}_n(\C))\otimes \Q
\]
Given a tuple $\overline{p} = (p_1,p_2,\ldots, p_n)$ of positive integers, we consider the finite dimensional C*-algebra associated to it by
\[
M(\overline{p}) := M_{p_1}(\C) \oplus M_{p_2}(\C) \oplus \ldots\oplus M_{p_n}(\C)
\]
By additivity of the functor $F_m$, we have
\[
F_m(M(\overline{p})) = \bigoplus_{j=1}^{n} \Q^{d(m,j)} \text{ where } d(m,j) = \begin{cases}
1 &: 0\leq m\leq 2p_j-1, m \text{ odd} \\
0 &: \text{ otherwise}
\end{cases}
\]
\end{example}

\section{AF-algebras}

An AF-algebra is a $C^{\ast}$-algebra which is an inductive limit of finite dimensional $C^{\ast}$-algebras. Given an AF-algebra $A$ and a specific inductive limit decomposition
\[
A_1\xrightarrow{\varphi_1} A_2 \xrightarrow{\varphi_2} A_3 \ldots \rightarrow A
\]
one can associate a diagram, called a Bratteli diagram, which encodes the algebras and maps that occur in this limit. We briefly review these ideas, and refer the reader to \cite[Section III.2]{davidson} for further details. \\

Let $\varphi : M(\overline{p^1})\to M(\overline{p^2})$ be a $\ast$-homomorphism between two finite dimensional C*-algebras (Using notation from \cref{examplefinite}). For $1\leq j\leq n_1$ and $1\leq i\leq n_2$, define $\varphi_{i,j} : M(p^1_j) \to M(p^2_i)$ to be the map given by
\[
M_{p^1_j}(\C) \hookrightarrow M(\overline{p^1}) \xrightarrow{\varphi} M(\overline{p^2}) \twoheadrightarrow M_{p^2_i}(\C)
\]
Then the multiplicity of $\varphi_{i,j}$ is
\begin{equation}\label{eqn:multiplicity}
\ell_{i,j} := \frac{Tr(\varphi_{i,j}(e))}{Tr(e)}
\end{equation}
where $e$ is any non-zero projection in $M(p^1_j)$. Note that this formula is independent of the choice of $e$. We write
\[
\Phi := [\ell_{i,j}] \in M_{n_2,n_1}(\Z^+)
\]
for the multiplicity matrix associated to $\varphi$. As is well known, the map $\varphi$ is unitarily equivalent to a map $\psi : A_1\to A_2$ which may be represented by a diagram $\mathcal{D}(A_1, A_2, \varphi)$ as
\begin{equation}\label{eqn:diagram}
\xymatrix{
p^1_1 \ar@2{-}[d]\ar@{-}[rd]\ar@3{-}[rd] & p^1_2\ar@{-}[d]\ar@{-}[rd] & \ldots & p^1_{n_1}\ar@{-}[d]\ar@2{-}[ld] \\
p^2_1 & p^2_2 & \ldots & p^2_{n_2}
}
\end{equation}
where the number of lines connecting $p^1_j$ to $p^2_i$ is $\ell_{i,j}$. Since the unitary group of a finite dimensional C*-algebra is connected, two unitarily equivalent $\ast$-homomorphisms are homotopic, and hence induce same maps at the level of $G_k(\cdot)$ and $F_m(\cdot)$. Thus, we shall henceforth assume that any $\ast$-homomorphism between finite dimensional $C^*$-algebras is given by a diagram as above. \\

Given an AF-algebra $A$ and a specific inductive limit decomposition
\[
A_1\xrightarrow{\varphi_1} A_2 \xrightarrow{\varphi_2} A_3 \ldots \rightarrow A
\]
one associates the diagram $\mathcal{D}(A)=\{\mathcal{D}(A_p,A_{p+1},\varphi_p)\}$, whose $p^{th}$ row is represented by tuples $\{(p,1) = p_1, (p,2) = p_2, \ldots, (p,n) = p_n\}$ if $A_p \cong M_{p_1}(\C)\oplus M_{p_2}(\C)\oplus \ldots \oplus M_{p_n}(\C)$, and the arrows from the $p^{th}$ row to the $(p+1)^{th}$ row encode the multiplicity matrix of $\varphi_p$. This is called a Bratteli diagram of $A$.\\

%\subsection{Computation of $F_m(A)$}
Our goal is to compute $F_m(A)$ from such a Bratteli diagram of $A$.  To this end, we need a few lemmas.

\begin{lemma}(\cite[II, Corollary 3.17]{mimura})\label{lem:inclusion_isomorphism}
For natural numbers $n<N$, let $\iota :M_n(\C)\hookrightarrow M_N(\C)$ denote the inclusion map $x\mapsto \text{diag}(x,0)$. Then, the induced map $G_k(\iota) : G_k(M_n(\C)) \to G_k(M_N(\C))$ is an isomorphism for $k<2n$ and a surjection for $k=2n$.
\end{lemma}

\begin{lemma}\label{inclusion}
For natural numbers $n<N$, let $\iota :M_n(\C)\hookrightarrow M_N(\C)$ denote the inclusion map $x\mapsto \text{diag}(x,0)$. Then, the induced map $F_m(\iota):F_m(M_n(\C))\rightarrow F_m(M_N(\C))$ is given by
\[
F_m(\iota) = \begin{cases}
\text{id} &: 1\leq m\leq 2n-1, m \text{ odd} \\
0 &: \text{ otherwise}
\end{cases}
\]
\end{lemma}
\begin{proof}
For simplicity we assume $N=n+1$, and consider the fibration $\mathcal{U}_n \hookrightarrow \mathcal{U}_{n+1} \to S^{2n+1}$, and the induced map $\iota^{\ast} : H^{\ast}(\mathcal{U}_{n+1};\mathbb{Q}) \to H^{\ast}(\mathcal{U}_{n};\mathbb{Q})$ given by
\[
\wedge(x_1,x_3,\hdots,x_{2n-1},x_{2n+1}) \xrightarrow{\iota^{\ast}} \wedge(y_1,y_3,\hdots,y_{2n-1})
\]
We now determine $\iota^{\ast}(x_m)$. For $m=2n+1$, by a theorem of Leray-Serre \cite[Theorem 5.2]{mccleary}, there is a first quadrant spectral sequence $\{E^{*,*}_r,d_r\}$ converging to $H^{\ast}(\mathcal{U}_{n+1};\Q)$, with 
\[
E^{p,q}_2\cong H^p(S^{2n+1};H^q(\mathcal{U}_n;\Q))
\]
Furthermore, by \cite[Example 1.C]{mccleary}, the sequence collapses at $E_2$, so we see that
\begin{equation*}
\begin{split}
H^{2n+1}(\mathcal{U}_{n+1};\mathbb{Q}) &= \bigoplus_{p+q=2n+1}E_\infty^{p,q} = \bigoplus_{p+q=2n+1}E_2^{p,q} \\
&= \bigoplus_{p+q=2n+1}H^p(S^{2n+1};\Q)\otimes H^q(\mathcal{U}_{n};\Q)) \\
&= H^{2n+1}(S^{2n+1};\Q)\bigoplus H^{2n+1}(\mathcal{U}_{n};\Q) \\
&= \Q x_{2n+1}\bigoplus H^{2n+1}(\mathcal{U}_{n};\mathbb{Q})
\end{split}
\end{equation*}
Thus, by \cite[Theorem 5.9]{mccleary}, $\iota^{\ast}:H^{2n+1}(\mathcal{U}_{n+1};\mathbb{Q}) \to H^{2n+1}(\mathcal{U}_{n};\mathbb{Q})$ is the projection map 
\[
\mathbb{Q}x_{2n+1}\bigoplus H^{2n+1}(\mathcal{U}_{n};\mathbb{Q})\twoheadrightarrow  E_\infty^{0,2n+1}= E_2^{0,2n+1}=H^{2n+1}(\mathcal{U}_{n};\mathbb{Q})
\]
Hence, $\iota^{\ast}(x_{2n+1})=0$. For $m\leq 2n-1$, by \cite[Example 1.C]{mccleary}, we have
\begin{equation*}
\begin{split}
H^m(\mathcal{U}_{n+1};\mathbb{Q}) &= {\bigoplus_{p+q=m}}E_\infty^{p,q} = \bigoplus_{p+q=m}E_2^{p,q} = \bigoplus_{p+q=m}H^p(S^{2n+1};\mathbb{Q})\otimes H^q(\mathcal{U}_n;\mathbb{Q}) \\
&= H^m(\mathcal{U}_n;\mathbb{Q})
\end{split}
\end{equation*}
Thus $\iota$ induces the identity map at the level of cohomology for $m\leq 2n-1$. We conclude that
\[
\iota^{\ast}(x_m) = \begin{cases}
y_m &: 0\leq m\leq 2n-1, m \text{ odd } \\
0 &: m=2n+1
\end{cases}
\]
By \cref{sullivan}, $\iota_{\ast} : \pi_m(\mathcal{U}_n)\otimes\mathbb{Q}\to \pi_m(\mathcal{U}_N)\otimes \mathbb{Q}$ is given by
\[
\iota_{\ast} = \begin{cases}
 \text{id} &: 1\leq m\leq 2n-1, m \text{ odd}\\
 0 &: \text{ otherwise}
 \end{cases}
\]
The result now follows from \cref{rem:quasi_unitary_unital}.
\end{proof}

The next lemma is well-known even for the functor $G_k$, but we include the proof for the sake of completeness.

\begin{lemma}\label{leq}
Fix integers $n,N,r\in \N$ such that $rn \leq N$, and consider the $\ast$-homomorphism $\eta :M_n(\C) \rightarrow M_N(\C)$ given by
\[
x\mapsto diag(\underbrace{x,x,x,\ldots, x}_{r\text{ times}},0)
\]
The map $F_m(\eta): F_m(M_n(\C))\to F_m(M_N(\C))$ satisfies
\[
F_m(\eta) = rF_m(\iota)
\] 
where $\iota : M_n(\C) \to M_N(\C)$ is the inclusion map $x\mapsto \text{diag}(x,0)$.
\end{lemma}
\begin{proof}
By \cref{rem:quasi_unitary_unital}, it suffices to show that $\eta_{\ast} = r \iota_{\ast}$ as maps between $\pi_m(\mathcal{U}_n(\C))$ and $\pi_m(\mathcal{U}_N(\C))$. To this end, let $[f]\in \pi_m(\mathcal{U}_n)$. Then, by Whitehead's Lemma \cite[Lemma 2.1.5]{rordam} applied to the algebra $M_N(C(S^m)\otimes M_n(\C))$, we have
\[
\eta_{\ast}([f])=[\text{diag}(\underbrace{f,f,\ldots,f}_{r\text{ times}},1)]=[diag(f^r,1)]
\]
Thus $\eta_{\ast}[f]=[\iota(f^r)]=\iota_{\ast}[f^r]=r\iota_{\ast}[f]$ as required.
\end{proof}

\begin{theorem}\label{finitedimthm}
For $k=1,2$, let $\overline{p^k} = (p^k_1, p^k_2, \ldots, p^k_{n_k})$ be two tuples of positive integers. Given a $\ast$-homomorphism $\varphi : M(\overline{p^1}) \to M(\overline{p^2})$ and $m\in \N$, we have
\[
F_m(M(\overline{p^k})) = \bigoplus_{j=1}^{n_k} \Q^{d(m,k,j)} \text{ where } d(m,k,j) = \begin{cases}
1 &: 0\leq m\leq 2p^k_j-1, m \text{ odd} \\
0 &: \text{ otherwise}
\end{cases}
\]
and $F_m(\varphi) : F_m(M(\overline{p^1})) \to F_m(M(\overline{p^2}))$ is given by multiplication by its multiplicity matrix $\Phi$.
\end{theorem}
\begin{proof}
The first part of the theorem is \cref{examplefinite}, so we compute $F_m(\varphi)$. Replacing $\varphi$ by a map that is unitarily equivalent (and hence homotopic) to $\varphi$, we may assume that $\varphi$ is represented by a diagram as in \cref{eqn:diagram}. Writing $\varphi=(f_1,f_2,\hdots,f_{n_k})$ where $f_i:M(\overline{p^1})\rightarrow M(\overline{p_i^2})$, it suffices to prove the theorem for a map 
\[
\varphi:M_{n_1}\oplus M_{n_2}\oplus\hdots\oplus M_{n_k}\rightarrow M_{\ell}
\]
given by $1\times k$ matrix $[r_1,r_2,\hdots,r_k]$. By \cref{leq}, it suffices to prove the result when $k=2$ and $r_1=r_2=1$. So we consider the map
\[
\varphi:M_{n_1}(\C) \oplus M_{n_2}(\C)\to  M_{\ell}(\C) \text{ given by } (x,y)\mapsto \text{diag}(x,y)
\]
where $n_1 \leq n_2$ and show that $F_m(\varphi)$ is given by multiplication by matrix $[1,1]$. To do this, consider the natural inclusion and projection maps
\begin{equation*}
\begin{split}
\iota_1 : M_{n_1}(\C) \hookrightarrow M_{\ell}(\C),\qquad  & \iota_2 : M_{n_2}(\C) \hookrightarrow M_{\ell}(\C) \\
p_1 : M_{n_1}(\C) \oplus M_{n_2}(\C) \to M_{n_1}(\C), \text{ and } & p_2 : M_{n_1}(\C) \oplus M_{n_2}(\C) \to M_{n_2}(\C)
\end{split}
\end{equation*}
Then $\varphi = (\iota_1\circ p_1)\cdot(\iota_2\circ p_2)$. Hence,
\[
F_m(\varphi)=F_m(\iota _1\circ p_1)+F_m(\iota_2\circ p_2) = F_m(\iota_1)F_j(p_1) + F_m(\iota_2)F_j(p_2)
\]
By \cref{inclusion}, it suffices to calculate $F_m(p_j)$. To begin with, we calculate $(p_j)^{\ast}$ between the respective cohomology rings. For this, note that
\begin{equation*}
\begin{split}
H^{\ast}(\mathcal{U}_{n_1};\Q)\xrightarrow{(p_1)^{\ast}} H^{\ast}(\mathcal{U}_{n_1}\oplus\mathcal{U}_{n_2};\Q) & \cong H^{\ast}(\mathcal{U}_{n_1};\Q)\otimes H^{\ast}(\mathcal{U}_{n_2}; \Q)\\
(p_{1})^{\ast}(a)\smile (p_{2})^{\ast}(b) & \leftrightarrow a\otimes b
\end{split}
\end{equation*}
where the above isomorphism is the Kunneth isomorphism, where $p_1^{\ast}(a)\mapsto a\otimes 1$. Thus $p_1^{\ast} : H^{\ast}(\mathcal{U}_{n_1};\mathbb{Q}) \to H^{\ast}(\mathcal{U}_{n_1};\Q)\otimes H^{\ast}(\mathcal{U}_{n_2};\Q)$ is given by
\begin{equation*}
\begin{split}
\wedge(x_1,x_3,\hdots,x_{2n_1-1}) &\to \wedge(x_1,x_3,\hdots,x_{2n_1-1})\otimes\wedge(y_1,y_3,\hdots,y_{2n_2-1})\\
x_i &\mapsto x_i\otimes 1 
\end{split}
\end{equation*}
for all $1\leq i\leq 2n_1-1$, where $i$ is odd. Thus by \cref{sullivan}, the associated linear map $(p_1)^{\ast}$ is given by
\[
(p_1)^{\ast}: \Q\to \Q\oplus\Q = \begin{cases}
r\mapsto (r,0) &: 0\leq i\leq 2n_1-1, i \text{ odd} \\
0 &: \text{ otherwise}
\end{cases}
\]
Dualizing, we get
\[
F_m(p_1)= (p_1)_{\ast}: \mathbb{Q}\oplus\mathbb{Q}\rightarrow\mathbb{Q} = \begin{cases}
(r,r')\mapsto r &: 0\leq m\leq 2n_1-1, m \text{ odd} \\
0 &: \text{ otherwise}
\end{cases}
\]
Similarly,
\[
F_m(p_2) = (p_2)_{\ast}: \mathbb{Q}\oplus\mathbb{Q}\rightarrow\mathbb{Q} = \begin{cases}
(r,r')\mapsto r' &: 0\leq m\leq 2n_2-1, m \text{ odd} \\
0 &: \text{ otherwise}
\end{cases}
\]
Thus combining the above two and the expressions for $F_m(\iota_j)$ from \cref{inclusion}, we get
\[
F_m(\varphi)=\begin{cases}
\mathbb{Q}\oplus\mathbb{Q}\to\mathbb{Q}&: 0\leq m\leq 2n_1-1, \text{ and odd}\\  
(r,r')\mapsto r+r' \\
\mathbb{Q}\to\mathbb{Q}&: 2n_1< m\leq 2n_2-1, \text{ and odd}\\  
r\mapsto r \\
0 &: \text{ otherwise}
\end{cases}
\]
Hence, $F_m(\varphi)$ is given by multiplication by matrix $[1,1]$ as required.
\end{proof}

Given a $\ast$-homomorphism $\varphi : A_1\to A_2$ between two finite dimensional C*-algebras, and $m \in \N$, we now get a diagram of $\Q$-vector spaces which represents the induced map $F_m(\varphi) : F_m(A_1) \to F_m(A_2)$:
\[
\xymatrix{
d(m,1,1) \ar@2{-}[d]\ar@3{-}[rd] & d(m,1,2)\ar@{-}[d]\ar@{-}[rd] & \ldots & d(m,1,n_1)\ar@{-}[d]\ar@2{-}[ld] \\
d(m,2,1) & d(m,2,2) & \ldots & d(m,2,n_2)
}
\]
where the number of lines connecting $d(m,1,j)$ to $d(m,2,i)$ is $\ell_{i,j}$, given by \cref{eqn:multiplicity}. We denote this diagram by
\[
\mathcal{D}_m(A_1, A_2, \varphi)
\]

\begin{example}
To illustrate the above result, we give an example. Let
\[
A_1 = \C\oplus M_2(\C)\oplus M_3(\C), \text{ and } A_2 = \C\oplus M_3(\C)\oplus M_5(\C)\oplus M_8(\C)
\]
and $\varphi : A_1\to A_2$ is given by
\[
\varphi(x,y,z) := \left(x,\begin{pmatrix}
x & 0 \\
0 & y
\end{pmatrix}, \begin{pmatrix}
x & 0 & 0 \\
0 & x & 0 \\
0 & 0 & z
\end{pmatrix} \begin{pmatrix}
y & 0 & 0 \\
0 & z & 0 \\
0 & 0 & z
\end{pmatrix} \right)
\]
Then $\mathcal{D}(A_1,A_2,\varphi)$ is 
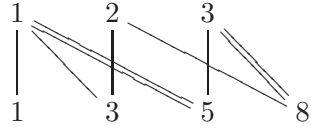
\[
\xymatrix{
1\ar@{-}[d]\ar@{-}[rd]\ar@2{-}[rrd] & 2\ar@{-}[d]\ar@{-}[rrd] & 3\ar@{-}[d]\ar@2{-}[rd] \\
1 & 3 & 5 & 8
}
\]
And the multiplicity matrix is
\[
\Phi = \begin{pmatrix}
1 & 0 & 0 \\
1 & 1 & 0 \\
2 & 0 & 1 \\
0 & 1 & 2
\end{pmatrix}
\]
We now consider the various diagrams $\mathcal{D}_m(A_1, A_2, \varphi)$:
\begin{itemize}
\item For $m\notin \{1,3,5\}$, the domain $F_m(A_1)$ is the zero vector space by \cref{examplefinite}, so the diagrams are not represented.
\item For $m=1$, $\mathcal{D}_1(A_1, A_2, \varphi)$ is
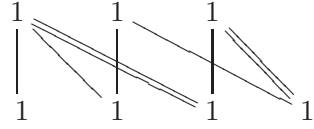
\[
\xymatrix{
1\ar@{-}[d]\ar@{-}[rd]\ar@2{-}[rrd] & 1\ar@{-}[d]\ar@{-}[rrd] & 1\ar@{-}[d]\ar@2{-}[rd] \\\
1 & 1 & 1 & 1
}
\]
Hence, $F_1(\varphi): \Q^3 \to \Q^4$ is given by $(a,b,c) \mapsto (a,a+b,2a+c,b+2c)$.
\item For $m=3$, $d(3,1,1) = d(3,2,1) = 0$, so $\mathcal{D}_3(A_1, A_2, \varphi)$ is
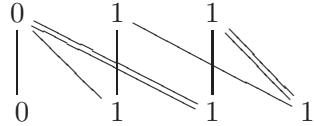
\[
\xymatrix{
0\ar@{-}[d]\ar@{-}[rd]\ar@2{-}[rrd] & 1\ar@{-}[d]\ar@{-}[rrd] & 1\ar@{-}[d]\ar@2{-}[rd] \\\
0 & 1 & 1 & 1
}
\]
so that $F_3(\varphi) : \Q^2 \to \Q^3$ is the map $(b,c) \mapsto (b,c,b+2c)$.
\item For $m=5$, $d(5,1,1)= d(5,1,2) = d(5,2,1) = 0 $, so $\mathcal{D}_5(A_1, A_2, \varphi)$ is
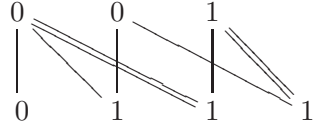
\[
\xymatrix{
0\ar@{-}[d]\ar@{-}[rd]\ar@2{-}[rrd] & 0\ar@{-}[d]\ar@{-}[rrd] & 1\ar@{-}[d]\ar@2{-}[rd] \\\
0 & 1 & 1 & 1
}
\]
so that $F_5(\varphi) : \Q\to \Q^3$ is the map $c \mapsto (0,c,2c)$.
\end{itemize}
\end{example}

Using the continuity of the functors $F_m(\cdot)$, \cref{mainthm:one} now follows.
 
\begin{theorem}
Given a labelled Bratteli diagram
\[
\mathcal{D}(A) = \{ (A_n, A_{n+1}, \varphi_n) :n\in \N\}
\]
defining an AF-algebra $A$, and given $m\in \N$, there is a corresponding diagram
\[
\mathcal{D}_m(A) = \{\mathcal{D}_m(A_n, A_{n+1}, \varphi_n) : n\in \N\}
\]
of $\Q$-vector spaces such that the inductive limit of the chain system defined by $\mathcal{D}_m(A)$ is isomorphic to $F_m(A)$.
\end{theorem}
 
We illustrate the above calculation with another example.

\begin{example}\label{example:af_calculation}
Consider the AF-algebra $A$ given by the Bratteli diagram below:
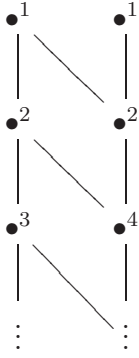
\[
\xymatrix{
\bullet^{1}\ar@{-}[d]\ar@{-}[rd] & \bullet^{1}\ar@{-}[d] \\
\bullet^{2}\ar@{-}[d]\ar@{-}[rd] & \bullet^{2}\ar@{-}[d] \\
\bullet^{3}\ar@{-}[d]\ar@{-}[rd] & \bullet^{4}\ar@{-}[d] \\
\vdots & \vdots
}
\]
We denote the algebra associated to each row by $A_k$, and the natural connecting map $A_k\to A_m$ by $\varphi_{m,k}$ for $m\geq k$. For any $m \in \N$ odd, there is a row $n$, such that, for all $k\geq n$,
\[
F_m(A_k) \cong \Q\oplus \Q
\]
Furthermore, the map $F_m(\varphi_{k+1,k}) : F_m(A_k) \to F_m(A_{k+1})$ is given by multiplication by the matrix
\[
\Phi = \begin{pmatrix}
1 & 0 \\
1 & 1
\end{pmatrix}
\]
Hence,
\[
F_m(A) \cong \begin{cases}
\Q\oplus \Q &: m \text{ odd } \\
0 &: m \text{ even}
\end{cases}
\]
\end{example}
\section{$K$-stability}

The notion of K-stability given below is due to Thomsen \cite[Definition 3.1]{thomsen}, and that of rational K-stability has been studied by Farjoun and Schochet \cite[Definition 1.2]{farjoun}, where it was termed rational Bott-stability.

\begin{definition}\label{defn:k_stable}
Let $A$ be a $C^*$-algebra and $j\geq 2$. Define $\iota_j: M_{j-1}(A)\to M_j(A)$ to be the natural inclusion map
\[
a\mapsto\begin{pmatrix}
a&0\\
0&0
\end{pmatrix}
\]
\begin{itemize}
\item $A$ is said to be K-stable if $G_k(\iota_j): G_k(M_{j-1}(A))\to G_k(M_j(A))$ is an isomorphism for all $k\geq 0$ and all $j\geq 2$.
\item $A$ is said to be rationally K-stable if $F_m(\iota_j):F_m(M_{j-1}(A))\to F_m(M_j(A))$ is an isomorphism for all $m\geq 1$ and all $j\geq 2$.
\end{itemize}
 
\end{definition}
Note that, for a $K$-stable C*-algebra, $G_k(A) \cong K_{k+1}(A)$ and for a rationally $K$-stable C*-algebra, $F_m(A) \cong K_{m+1}(A)\otimes \Q$. Clearly, $K$-stability implies rational $K$-stability. We now show for the class of $AF$-algebras both these notions are equivalent. \\

Once again, we do this using a Bratteli diagram. To this end, we need the following notion. For a finite dimensional C*-algebra $A$, write
\[
\min\dim(A) := \min\{\text{square root of the dimension of a simple summand of } A\}
\]
In other words, if $A \cong M(\overline{p})$ for a tuple $\overline{p} = (p_1,p_2,\ldots, p_n)$ of positive integers, then $\min\dim(A) = \min\{p_j\}$. Note that this number may also be defined intrinsically using extremal traces and minimal central projections, but we give this definition for the sake of simplicity. \\

The next result is a simple consequence of \cref{examplefinite} and the continuity of the functor $F_m$.

\begin{lemma}\label{lem:even_zero}
Let $A$ be an AF-algebra and $m\in \N$ be an even number, then $F_m(A) = 0$.
\end{lemma}

\begin{lemma}\label{lem:exact}
Let $m \in \N$ be an odd number, then the functor $F_m$ is exact on the class of AF-algebras.
\end{lemma}
\begin{proof}
Given a short exact sequence $0 \to I \to A\to A/I \to 0$ of AF-algebras, we consider the long exact sequence in the $G_m$ functors from \cite[Theorem 2.5]{thomsen}. Tensoring with $\Q$, the long exact sequence remains exact. Now observe that if $m\in \N$ is odd, then
\[
F_{m+1}(A/I) = F_{m-1}(I) = 0
\]
by \cref{lem:even_zero}. Hence the result.
\end{proof}

\begin{lemma}\label{lem:injective}
For any AF-algebra $A$ and $n\in \N$, the natural inclusion $\iota_A : A\to M_n(A)$ induces an injective map
\[
F_m(\iota_A) : F_m(A) \to F_m(M_n(A))
\]
\end{lemma}
\begin{proof}
If $A = M_k(\C)$, then
\[
F_m(A) = \begin{cases}
\Q &: 0\leq m\leq 2k-1, j \text{ odd} \\
0 &: \text{ otherwise}
\end{cases}
\]
And if $0\leq m\leq 2k-1$, the map $(\iota_A)_{\ast} : \pi_m(\mathcal{U}_k(\C)) \to \pi_m(\mathcal{U}_{nk}(\C))$ is an isomorphism. Hence, $F_m(\iota_A)$ is an isomorphism for $0\leq m\leq 2k-1$, and is clearly injective otherwise. By additivity, the result is true if $A$ is finite dimensional. \\

Now if $A$ is an infinite dimensional AF-algebra, then the result follows by continuity of the functor $F_m$. To see this, write $A = \lim (A_k, \varphi_k)$, where $A_k$ are finite dimensional with connecting maps $\varphi_{\ell,k} : A_k\to A_{\ell}$ for $\ell \geq k$, and let $\alpha_k : A_k\to A$ denote $\ast$-homomorphisms such that $\alpha_{k+1}\circ \varphi_k = \varphi_{k+1}$ for all $k\in \N$. Then by continuity of the functor $F_m$, we have
\[
F_m(A) \cong \lim (F_m(A_k), F_m(\varphi_k))
\]
So if $x\in F_m(A)$ such that $F_m(\iota_A)(x) = 0$, there exists $k\in \N$ and $y\in F_m(A_k)$ such that $x = F_m(\alpha_k)(y)$. Let $\alpha_k^{(n)} : M_n(A_k)\to M_n(A)$ denote the inflation of $\alpha_k$, then $\alpha_k^{(n)}\circ \iota_{A_k} = \iota_A \circ \alpha_k$. Hence,
\[
F_m(\alpha_k^{(n)})\circ F_m(\iota_{A_k})(y)= F_m(\iota_A)\circ F_m(\alpha_k)(y) = F_m(\iota_A)(x) = 0
\]
But $F_m(M_n(A)) \cong \lim (F_m(M_n(A_k)), F_m(\varphi_k^{(n)}))$, so there exists $\ell\geq k$ such that $F_m(\varphi_{\ell,k}^{(n)})\circ F_m(\iota_{A_k})(y) = 0$. But $\varphi_{\ell,k}^{(n)}\circ \iota_{A_k} = \iota_{A_{\ell}}\circ \varphi_{\ell,k}$. This implies
\[
F_m(\iota_{A_{\ell}})\circ F_m(\varphi_{\ell,k})(y) = 0
\]
But $A_{\ell}$ is finite dimensional, so $F_m(\iota_{A_{\ell}})$ is injective by the first part of the proof. Hence, $F_m(\varphi_{\ell,k})(y) = 0$, which implies
\[
x = F_m(\alpha_k)(y) = F_m(\alpha_{\ell})\circ F_m(\varphi_{\ell,k})(y) = 0
\]
Thus, $F_m(\iota_A)$ is injective as required.
\end{proof}

Let $A$ be an AF-algebra given as an inductive limit
\[
A_1 \xrightarrow{\varphi_1} A_2 \xrightarrow{\varphi_2} A_3\to \ldots \to A
\]
where each $A_p$ is finite dimensional. If each $\varphi_p$ is injective, then we say that $\{A_p\}$ is a \textit{generating nest} of finite dimensional C*-algebras.  \\

Let $\mathcal{D}(A)$ be the corresponding Bratteli diagram. For each $m\in \N$, we write $\mathcal{D}(A;m)$ for the $m^{th}$ level of $\mathcal{D}(A)$ (corresponding to the algebra $A_m$). For a node $(p,i) \in \mathcal{D}(A)$,  we write $(p,i) = n$ if the $i^{th}$ summand of $A_p$ is $M_n(\C)$. For two nodes $(p,i)$ and $(p+1,j) \in \mathcal{D}(A)$, we write $(p,i)\searrow (p+1,j)$ is there is an edge connecting these two nodes. If this happens, we say that $(p+1,j)$ is a successor of $(p,i)$, and that $(p,i)$ is a predecessor of $(p+1,j)$. We say that a node $(p,i)$ is an \textit{orphan} if it has no predecessors (in other words, there does not exist any $j$ such that $(p-1,j)\searrow (p,i)$).

\begin{definition}\label{defn: k_chain}
Let $K \in \N$ be a positive integer. A subset $\Lambda \subset \mathcal{D}(A)$ is called a \textit{$K$-chain} if there exists $M \in \N$ and $N \in \N\cup\{+\infty\}$ such that the following conditions hold:
\begin{enumerate}
\item $\Lambda\cap \mathcal{D}(A;j) = \emptyset$ for all $j<M$ and $j>N$. Furthermore, $|\Lambda\cap \mathcal{D}(A;j)| = 1$ for all $M\leq j\leq N$. Write $\Lambda = \{(p,i_p) : M\leq p\leq N\}$.
\item For each $M\leq p\leq N, (p,i_p)\searrow (p+1,i_{p+1})$.
\item Furthermore, if $(p,i)\searrow (p+1,i_{p+1})$ for any $i$, then $i=i_p$.
\item $(p,i_p) = K$ for all $M\leq p\leq N$.
\end{enumerate}
If $N < \infty$, then $\Lambda$ is said to be a \textit{terminating} $K$-chain. If $N = +\infty$, then $\Lambda$ is said to be an \textit{infinite} $K$-chain. 
\end{definition}

We should mention that condition (3) of the above definition is crucial. It says that every node in $\Lambda$ (barring the first one) has precisely one predecessor in $\mathcal{D}(A)$, and that predecessor must also be in $\Lambda$.

\begin{lemma}\label{lem: k_chain}
Let $\mathcal{D}(A)$ be a Bratteli diagram associated to an AF-algebra $A$ and $K \in \N$. If $\mathcal{D}(A)$ has an infinite $K$-chain, then $A$ has an ideal $I$ such that $A/I\cong M_K(\C)$.
\end{lemma}
\begin{proof}
Assume without loss of generality that $M=1$ in the definition of $\Lambda$, and set $\mathcal{S} := \mathcal{D}(A)\setminus \Lambda$. Then we claim that $\mathcal{S}$ is a directed, hereditary set. To see that $\mathcal{S}$ is directed, let $(p,i) \in \mathcal{S}$ and $(p,i)\searrow (p+1,j)$. If $(p+1,j)\in \Lambda$, then then by condition (3) of \cref{defn: k_chain} , it follows that $(p,i) \in \Lambda$ as well. This is a contradiction, so $(p+1,j)\in \mathcal{S}$ as well. Hence, $\mathcal{S}$ is directed. To see that $\mathcal{S}$ is hereditary, let $(p,i) \in \mathcal{D}$ such that, for any $j, (p,i)\searrow (p+1,j)$ implies that $(p+1,j) \in \mathcal{S}$. Then, we wish to prove that $(p,i) \in \mathcal{S}$. Suppose not, then $i=i_p$. In that case, by condition (2), $(p,i)\searrow (p+1,i_{p+1})$, so $(p+1,i_{p+1}) \in \mathcal{S}$. This is a contradiction. Hence, $(p,i) \in \mathcal{S}$ as required. \\

Hence, there is an ideal $I\vartriangleleft A$ such that the diagram of $A/I$ is given by $\Lambda$. However, by conditions (1) and (4), each node of $\Lambda$ represents the algebra $M_K(\C)$. Upto unitary equivalence, there is only one map $M_K(\C)\to M_K(\C)$, namely the identity map. Hence, $A/I \cong M_K(\C)$.
\end{proof}

Before we proceed, we fix some notation: Let $C$ be a finite dimensional C*-algebra, and $j\in \N$. Define $C^{(j)}$ to be the direct sum of all simple summands of $C$ of dimension equal to $j^2$ (We adopt the convention that the direct sum over an empty index set is the zero C*-algebra). Similarly, $C^{(>j)}$ denotes the direct sum of all simple summands of $C$ whose dimension is $>j^2$, and $C^{(<j)}$ is the direct sum of all simple summands of $C$ whose dimension is $<j^2$.  Hence,
\[
C = C^{(<j)}\oplus C^{(j)}\oplus C^{(>j)}
\] 
Given a diagram $\mathcal{D}(A) =\{(A_n, A_{n+1}, \varphi_n): n\in \N\}$, a node $(p,i) \in \mathcal{D}(A)$ and a level $m\geq p$, we write
\[
(p,i)\hookrightarrow A_m^{(>j)}
\]
if every arrow emanating from $(p,i)$ lands in a node corresponding to a summand of $A_m^{(>j)}$. Finally, given a Bratteli diagram $\mathcal{D}(A)$ as above and integers $m\geq n$, we write $\varphi_{m,n} := \varphi_{m-1}\circ \varphi_{m-2}\circ \ldots \circ\varphi_n : A_n\to A_m$ with the convention that $\varphi_{n,n} = \text{id}_{A_n}$. \\

\begin{lemma}\label{lem: min_dim_large}
Let $A$ be an AF-algebra with no non-zero finite dimensional representations. Then, for each $m \in \N$, there is a generating nest $\{A_{m,p} : p \in \N\}$ of finite dimensional C*-algebras such that
\[
\min\dim(A_{m,p}) \geq m
\]
for all $p \in \N$.
\end{lemma}
\begin{proof}
We begin with any Bratteli diagram associated to $A$
\[
\mathcal{D}^{(1)}(A) = \{(A_{1,p}, A_{1,p+1}, \varphi^{(1)}_p) : p\in \N\}
\]
and we inductively construct diagrams
\[
\mathcal{D}^{(m)}(A) = \{(A_{m,p}, A_{m,p+1}, \varphi^{(m)}_p): p\in \N\}
\]
such that, for each $m, p\in \N$, $\min\dim(A_{m,p}) \geq m$, and the connecting maps $\varphi^{(m)}_p$ are all injective.\\

Since $\min\dim(A_{1,p}) \geq 1$ for all $p\in \N$, we assume that we have constructed $\mathcal{D}^{(i)}(A)$ for $1\leq i\leq m$, and we now construct $\mathcal{D}^{(m+1)}(A)$. If $\min\dim(A_{m,p}) \geq m+1$ for all but finitely many $p\in \N$, then we simply take $\mathcal{D}^{(m+1)}(A) = \mathcal{D}^{(m)}(A)$, by ignoring the first finitely many terms and appropriately relabelling the objects and maps. \\

Therefore, we assume without loss generality that $\min\dim(A_{m,p}) = m$ for infinitely many values of $p$, and that $\min\dim(A_{m,1}) = m$. Now consider the nodes appearing in $A_{m,1}^{(m)}$, and enumerate them as
\[
(1,1), (1,2), \ldots, (1,s)
\]
For each $1\leq i\leq s$, we may begin a $m$-chain starting at $(1,i)$. Since $\min\dim(A_{m,j})\geq m$ for all $j\in \N$ and all the connecting maps are injective, any such chain will satisfy condition (3) of \cref{defn: k_chain}. Since $A$ does not have any non-zero finite dimensional representations, any such chain must terminate by \cref{lem: k_chain}. Hence, there exists $m_i \in \N$ such that $(1,i)\hookrightarrow A_{m,m_i}^{(>m)}$. Since the connecting maps are injective, this implies that
\[
(1,i)\hookrightarrow A_{m,k}^{(>m)}
\]
for all $k\geq m_i$. Since $\min\dim(A_{m,p}) = m$ for infinitely many values of $p\in \N$, we choose $n_1 \geq \max\{m_1, m_2, \ldots, m_s\}$ such that $A_{m,n_1}^{(m)}\neq 0$. Then,
\[
(1,i)\hookrightarrow A_{m,n_1}^{(>m)}
\]
for all $1\leq i\leq s$, and therefore every node of $A_{m,n_1}^{(m)}$ is an orphan. Starting with these nodes, and repeating the above argument, we obtain a subsequence $\{n_j\}_{j=1}^{\infty}$ such that, for each $j\in \N$, every node of $A_{m,n_j}^{(m)}$ is an orphan. Consider $\iota_j : A_{m,n_j}^{(>m)}\to A_{m,n_j}$ and $\pi_j : A_{m,n_j}\to A_{m,n_j}^{(>m)}$ to be the natural inclusion and quotient maps respectively. Then, it follows that
\[
\iota_j\circ \pi_j\circ \varphi^{(m)}_{n_j,n_{j-1}} = \varphi^{(m)}_{n_j,n_{j-1}}
\] 
for all $j\geq 1$, with the convention that $n_0 = 1$. Hence, the following diagram commutes
\[
\xymatrix{
\ldots\ar[r] & A_{m,n_{j-1}}\ar[rr]^{\varphi^{(m)}_{n_j,n_{j-1}}}\ar[d]_{\pi_j\circ\varphi^{(m)}_{n_j,n_{j-1}}} && A_{m,n_j}\ar[d]^{\pi_{j+1}\circ\varphi^{(m)}_{n_{j+1},n_j}}\ar[r]^{\varphi^{(m)}_{n_{j+1},n_j}} & \ldots \\
\ldots\ar[r] & A_{m,n_j}^{(>m)}\ar[rr]_{\pi_{j+1}\circ\varphi_{n_{j+1},n_j}\circ \iota_j}\ar[rru]^{\iota_j} && A_{m,n_{j+1}}^{(>m)}\ar[r] & \ldots
}
\]
We set $(A_{m+1,j}, \varphi^{(m+1)}_j)$ to be the terms in the lower row. Then, it follows from \cite[Exercise 6.8]{rordam}, that $\lim (A_{m+1,j}, \varphi^{(m+1)}_j) \cong A$. Furthermore, by construction, we have $\min\dim(A_{m+1,j})\geq m+1$ for all $j\in \N$. Finally, $\iota_{j+1}\circ \varphi^{(m+1)}_j = \varphi^{(m)}_{n_{j+1},n_j}\circ \iota_j$. Since each $\varphi^{(m)}_k$ is assumed to be injective, it follows that $\varphi^{(m+1)}_j$ is injective as well.
\end{proof}

We are now in a position to prove \cref{mainthm:two}.

\begin{theorem}\label{thm:k_stable}
For an $AF$-algebra $A$, the following are equivalent:
\begin{enumerate}
\item $A$ is $K$-stable
\item $A$ is rationally $K$-stable
\item $A$ has no non-zero finite dimensional representations.
\item For each $m \in \N$, there is a generating nest $\{A_{m,p} : p \in \N\}$ of finite dimensional C*-algebras such that
\[
\min\dim(A_{m,p}) \geq m
\]
for all $p \in \N$.
\end{enumerate}
\end{theorem}
\begin{proof} Note that $(1)\Rightarrow (2)$ is obvious, and $(3)\Rightarrow (4)$ is the content of \cref{lem: min_dim_large}. Therefore, we prove $(2)\Rightarrow (3)$ and $(4)\Rightarrow (1)$. \\

$(2)\Rightarrow (3)$: Suppose $A$ has a non-zero finite dimensional representation, then there is a proper ideal $I< A$ such that $Q := A/I$ is finite dimensional. By taking a further quotient if need be, we may assume without loss of generality that $Q = M_K(\C)$ for some $K\in \N$. Now fix $n > 1$ and consider the commuting diagram
\[
\xymatrix{
0\ar[r] & I\ar[r]\ar[d] & A\ar[r]\ar[d] & Q\ar[d]\ar[r] & 0 \\
0\ar[r] & M_n(I)\ar[r] & M_n(A)\ar[r] & M_n(Q)\ar[r] & 0
}
\]
Choose an odd number $m$ such that
\[
2K - 1 < m \leq 2nK-1
\]
and applying the functor $F_m$ to this diagram, we get an induced diagram of $\Q$-vector spaces by \cref{lem:exact}
\[
\xymatrix{
0\ar[r] & F_m(I)\ar[r]\ar[d]^{F_m(\iota_I)} & F_m(A)\ar[r]\ar[d]^{F_m(\iota_A)} & 0\ar[d]^{F_m(\iota_Q)}\ar[r] & 0 \\
0\ar[r] & F_m(M_n(I))\ar[r] & F_m(M_n(A))\ar[r] & \Q \ar[r] & 0
}
\]
Since $F_m(\iota_Q)$ is the zero map (and hence not surjective), it follows that $F_m(\iota_A)$ cannot be surjective. This contradicts the assumption that $A$ is rationally $K$-stable. \\

$(4)\Rightarrow (1)$: Fix $k\in \N\cup\{0\}$ and $j\geq 2$. By hypothesis, there is a generating nest $\{A_p : p\in \N\}$ of finite dimensional subalgebras such that
\[
\min\dim(A_p)\geq \left\lceil \frac{k+1}{2}\right\rceil
\]
for all $p\in \N$. We want to show that the map $G_k(\iota_j) : G_k(M_{j-1}(A))\to G_k(M_j(A))$ is an isomorphism. Now, $M_{j-1}(A)$ is an inductive limit of the algebras $\{M_{j-1}(A_p)\}$ which also satisfy the condition that
\[
\min\dim(M_{j-1}(A_p))\geq \left\lceil \frac{k+1}{2}\right\rceil
\]
for all $p\in \N$. Therefore, replacing $M_{j-1}(A)$ by $A$, it suffices to show that, for each $n\in \N$, the inclusion map $\iota_A : A\to M_n(A)$ induces an isomorphism $G_k(\iota_A)$. \\

Now fix $n\in \N$. Then, for each simple summand $M_{\ell}(\C) \hookrightarrow A_p$, we have $k\leq 2\ell-1$, so that the inclusion map $\iota_{M_{\ell}(\C)} : M_{\ell}(\C) \to M_{n\ell}(\C)$ induces an isomorphism $G_k(\iota_{M_{\ell}(\C)}) : G_k(M_{\ell}(\C))\to G_k(M_{n\ell}(\C))$ by \cref{lem:inclusion_isomorphism}. Hence, $G_k(\iota_{A_p})$ is an isomorphism for all $p\in \N$. By taking limits, we conclude that $G_k(\iota_A)$ is also an isomorphism as required.
\end{proof}

We now connect our results to those of Thomsen \cite{thomsen}. Recall that an ordered Abelian group $(G,G^+)$ is said to have \textit{large denominators} if, for all $a\in G^+$ and $n\in \N$, there exists $b\in G^+$ and $m\in \N$ such that
\[
nb\leq a \leq mb
\]
In \cite[Theorem 4.5]{thomsen}, Thomsen proves that, if $A$ is an AF-algebra such that $K_0(A)$ has large denominators, then $A\otimes B$ is $K$-stable for all C*-algebras $B$. Using \cref{mainthm:two}, we are partially able to recover this result.

\begin{corollary}{\cite[Theorem 4.5]{thomsen}}
If $A$ is an AF-algebra such that $K_0(A)$ has large denominators, then $A$ is $K$-stable.
\end{corollary}
\begin{proof}
Suppose $\varphi : A\to M_K(\C)$ is a non-zero finite dimensional representation, then
\[
K_0(\varphi) : (K_0(A),K_0(A)^+) \to (\Z,\Z^+)
\]
is a homomorphism of ordered Abelian groups. Furthermore, since $\varphi \neq 0$, $K_0(\varphi) \neq 0$ by \cite[Theorem 1.3.4]{rordam_stormer}. Hence, there exists $a\in K_0(A)^+$ such that $k:= K_0(\varphi)(a) > 0$. By hypothesis, there exists $b\in K_0(A)$ and $m\in \N$ such that
\[
(k+1)b \leq a \leq mb
\]
Hence, $(k+1)K_0(\varphi)(b) \leq k$, which implies $K_0(\varphi)(b) = 0$. However, this contradicts the fact that $mK_0(\varphi)(b)\geq k$. Consequently, $A$ does not have any non-zero finite dimensional representations, and is hence $K$-stable by \cref{mainthm:two}.
\end{proof}

\begin{example} Consider the AF-algebra $A$ given by the Bratteli diagram in \cref{example:af_calculation}. The sub-diagram consisting of the column on the right defines an ideal $I$, so this algebra is not simple. In fact, $K_0(A)$ does not have large denominators. To see this, we denote the maps $A_k\to A$ by $\alpha_k$. Now, let $e \in A_1$ denote a minimal projection in the first summand $M_1(\C)$ of $A_1$. and let $a:= [\alpha_1(e)] \in K_0(A)$. Suppose that there exists $b\in K_0(A)^{+}$ and $m\in \N$ such that
\[
2b \leq a \leq mb
\]
Then choose $k\in \N$ and a projection $p \in A_k$ such that $b = [\alpha_k(p)]$. Also, $A_k = M_k(\C)\oplus M_{n_k}(\C)$ where $n_k = 1+\frac{k(k-1)}{2}$. So choose a normalized extremal trace $\tau$ on the first summand of $A_k$, then it follows that
\[
2\tau(p) \leq \tau(\varphi_{k,1}(e)) \leq m\tau(p)
\]
Now by construction $\tau(\varphi_{k,1}(e)) = \frac{1}{k}$. Hence, $\tau(p) \leq \frac{1}{2k}$, which implies that $\tau(p) = 0$. This cannot happen because $m\tau(p)\geq \frac{1}{k}$ as well. Thus, $K_0(A)$ does not have large denominators. \\

However, $A$ is $K$-stable because condition (4) of \cref{thm:k_stable} holds.
\end{example}

\subsection*{Acknowledgements}
The first named author is supported by UGC Junior Research Fellowship No. 1229 and the second named author was partially supported by SERB Grant YSS/2015/001060. The authors would like to thank the referee for a careful reading of the manusript, and for pointing out an error in an earlier version of the paper.

%\printbibliography
\end{document}